\documentclass[final]{dmtcs-episciences}
\usepackage[utf8]{inputenc}
\usepackage{amsmath, amssymb}
\usepackage[round]{natbib}
\usepackage[ruled,vlined,linesnumbered,resetcount,algosection]{algorithm2e}
\usepackage{tikz}
\pgfdeclarelayer{edgelayer}
\pgfdeclarelayer{nodelayer}
\pgfsetlayers{edgelayer,nodelayer,main}

\newtheorem{theorem}{Theorem}[section]
\newtheorem{definition}[theorem]{Definition}

\newtheorem{lemma}[theorem]{Lemma}
\newtheorem{proposition}[theorem]{Proposition}

\newtheorem{problem}[theorem]{Problem}
\newtheorem{observation}[theorem]{Observation}

\DeclareMathOperator{\BoxProduct}{\mathbin{\Box}}
\DeclareMathOperator{\diam}{diam}
\DeclareMathOperator{\rad}{rad}
\DeclareMathOperator{\e}{ecc}

\newcommand{\svSpan}[1]{\sigma^{\boxtimes}_V(#1)}
\newcommand{\seSpan}[1]{\sigma^{\boxtimes}_E(#1)}
\newcommand{\dvSpan}[1]{\sigma^{\times}_V(#1)}
\newcommand{\deSpan}[1]{\sigma^{\times}_E(#1)}
\newcommand{\cvSpan}[1]{\sigma^{\BoxProduct}_V(#1)}
\newcommand{\ceSpan}[1]{\sigma^{\BoxProduct}_E(#1)}

\author[Iztok Bani\v{c} and Andrej Taranenko]{Iztok Bani\v{c}\affiliationmark{1,2,3}
  \and Andrej Taranenko\affiliationmark{1,2}}

\title[Span of a graph: keeping the safety distance]{Span of a graph: keeping the safety distance}

\affiliation{
  University of Maribor, Faculty of Natural Sciences and Mathematics, Slovenia\\
  Institute of Mathematics, Physics and Mechanics, Slovenia\\
  University of Primorska, Andrej Maru\v si\v c Institute, Slovenia}


\begin{document}
\publicationdata{vol. 25:1}{2023}{8}{10.46298/dmtcs.9859}{2022-07-29; 2022-07-29; 2022-12-14}{2023-02-19}
\maketitle

\begin{abstract}
Inspired by Lelek's idea from [Disjoint mappings and the span of spaces, Fund. Math. 55 (1964), 199 -- 214], we introduce the novel notion of the span of graphs. Using this, we solve the problem of determining the \emph{maximal safety distance} two players can keep at all times while traversing a graph. Moreover, their moves must be made with respect to certain move rules. For this purpose, we introduce different variants of a span of a given connected graph. All the variants model the maximum safety distance kept by two players in a graph traversal, where the players may only move with accordance to a specific set of rules, and their goal: visit either all vertices, or all edges. For each variant, we show that the solution can be obtained by considering only connected subgraphs of a graph product and the projections to the factors. We characterise graphs in which it is impossible to keep a positive safety distance at all moments in time. Finally, we present a polynomial time algorithm that determines the chosen span variant of a given graph.
\end{abstract}

\section{Introduction}
In the times of the global pandemic which we have witnessed starting in 2020, two of the basic public safety measures that were introduced worldwide were social distancing and keeping a safety distance in public spaces. In this paper, we solve the problem of computing a maximal possible safety distance two people can keep at all times. 

Our concept is based on Lelek's span of a continuum which is introduced in \cite{Lelek}. This notion became extremely popular in continuum theory and many papers appeared, for an example see \cite{Hoehn}, where more references can be found. In \cite{Hoehn}, Hoehn proved that there are continua with zero span that are not chainable, which solved one of the most famous open problems in continuum theory. In the present paper, we show that this is not the case in the graph theoretical equivalent of the span, i.e., among other things we show that paths are the only graphs with zero span.

Imagine two players, say Alice and Bob, moving through a graph. They would both like to visit all vertices and/or all edges of the graph whilst keeping the maximum possible safety distance from each other. One way to describe the players' movements is to represent their positions at a fixed moment $t$ in time by a pair $(a_t, b_t)$, where both $a_t$ and $b_t$ are vertices of the graph. After that either one or both players can choose to move to an adjacent vertex or stay at the current one. 

Figure \ref{fig:motivation} shows an example of the graph $G$, and the location (the walk) of both players at five consecutive moments in time, shown by the graph $W$. For each moment in time Alice's location and Bob's location are represented by the red and the blue arrow, respectively. At the time $t_0$, Alice is at the vertex $r_1$ and Bob is at the vertex $r_3$, so their locations can be represented by the pair $(r_1, r_3)$. At this specific time, the players are keeping safe at the distance 2. Next, at the time $t_1$, Alice moves to the vertex $r_2$, whilst Bob stays at the vertex $r_3$ (this can be represented by the pair $(r_2, r_3)$, and the players are at the distance 1). At the time $t_2$, Alice moves to the vertex $r_3$ and Bob moves to the vertex $r_4$; thus obtaining the pair $(r_3, r_4)$ and maintaining the safety distance 1.
Similarly, at the time $t_3$, Alice stays at the vertex $r_3$ and Bob moves to the vertex $r_2$; thus obtaining the pair $(r_3, r_2)$ and maintaining the safety distance 1. Finally, at the time $t_4$, Alice moves to the vertex $r_4$, whilst Bob moves to the vertex $r_1$ (this can be represented by the pair $(r_4, r_1)$, and the players are at the distance 2). So in this example the walk of both players at five consecutive points in time can be represented by the tuple $((r_1,r_3), (r_2,r_3), (r_3, r_4), (r_3, r_2), (r_4, r_1))$. Note, in this example all the vertices are visited by both players and the maximal distance they were able to maintain at all times was one.

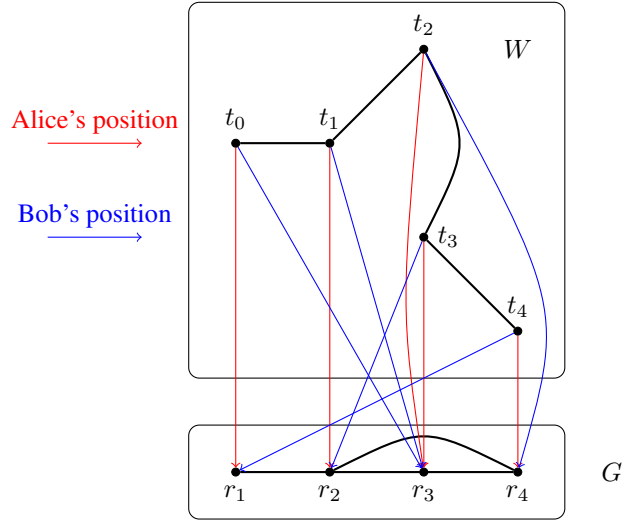
\begin{figure}[!ht]
	\centering
\begin{tikzpicture}[scale=1.25]
\tikzstyle{rn}=[circle, fill=black,draw, inner sep=0pt, minimum size=3pt]
\tikzstyle{loc}=[circle, fill=none,draw, inner sep=0pt, minimum size=5pt, red]
\tikzstyle{locb}=[circle, fill=none,draw, inner sep=0pt, minimum size=7pt, blue]

\node [style=rn] (r1) at (1, -0.5) [label=below:{$r_1$}] {};
\node [style=rn] (r2) at (2, -0.5) [label=below:{$r_2$}] {};
\node [style=rn] (r3) at (3, -0.5) [label=below:{$r_3$}] {};
\node [style=rn] (r4) at (4, -0.5) [label=below:{$r_4$}] {};
\draw [thick] (r2) .. controls (3, 0) .. (r4);
\draw [thick] (r1)--(r4);


\node [style=rn] (t0) at (1, 3) [label=above:{$t_0$}] {};
\node [style=rn] (t1) at (2, 3) [label=above:{$t_1$}] {};
\node [style=rn] (t2) at (3, 4) [label=above:{$t_2$}] {};
\node [style=rn] (t3) at (3, 2) [label=right:{$t_3$}] {};
\node [style=rn] (t4) at (4, 1) [label=above:{$t_4$}] {};

\draw [thick] (t0)--(t1)--(t2).. controls (3.5 , 3) .. (t3)--(t4);

\draw [rounded corners](0.5, 0.5) rectangle (4.5, 4.5);
\node at (4, 4) {$W$};

\draw [rounded corners](0.5, 0) rectangle (4.5, -1);
\node at (5, -0.5) {$G$};

\draw[red,->] (t0)--(r1);
\draw[red,->] (t1)--(r2);
\draw[red,->] (t2)..controls (2.75, 1.25) ..(r3);
\draw[red,->] (t3)--(r3);
\draw[red,->] (t4)--(r4);

\draw[blue,->] (t0)--(r3);
\draw[blue,->] (t1)--(r3);
\draw[blue,->] (t2)..controls (4.5, 1.25) ..(r4);
\draw[blue,->] (t3)--(r2);
\draw[blue,->] (t4)--(r1);

\draw[red,->] (-1,3)--node[midway, above] {Alice's position} (0,3);
\draw[blue,->] (-1,2)--node[midway, above] {Bob's position}(0,2);

\end{tikzpicture}
\caption{An example of two visitors' walks in a gallery.}
		\label{fig:motivation}
\end{figure}

In general, we are interested in keeping the maximal possible safety distance between both players amongst all possible walks through the given graph $G$. For a given connected graph $W$ that represents a walk (consecutive points in time are presented by adjacent vertices) the two mappings from $V(W)$ to $V(G)$ which show the location of the corresponding player must map adjacent vertices to either the same vertex (meaning the player did not move) or to adjacent vertices (meaning the player moved to an adjacent vertex). Such mappings are called \emph{weak homomorphisms}. 

Moreover, we assume that both players desire to visit all vertices and/or all edges of the graph. In what follows, we present a formal notion for the described situations for three different sets of movement rules at any observed point in time:
\begin{description}
    \item[traditional movement rules:] both players are independently allowed to move to an adjacent vertex or stay at their current location, 
    \item[active movement rules:] both players must move to an adjacent vertex, or
    \item[lazy movement rules:] exactly one player is allowed to move to an adjacent vertex.
\end{description}

Our problem of keeping a safety distance between two players is reminiscent of the classic cop and robber game on graphs, see \cite{CopsAndRobbers} for the book on the topic and for more references. In such a game, a cop (or several) and a robber are put on a graph. At each point in time all parties involved can move according to the movement rules. The goal is that the cop captures the robber in a finite number of steps. In our somewhat dual concept to this one, two players desire to maintain the maximal possible safety distance at all times. Also, in similar types of dynamic graph parameters the two players usually follow opposite goals (e.g. the cop wants to catch the robber, the robber does not want to be caught). In our problem, both players have the same goal of keeping the maximum possible safety distance.

We proceed as follows. In Section \ref{sec:preliminary} basic definitions and notations are presented. In Section \ref{dva} we define different vertex and edge span variants of a graph and prove that each span can be obtained from a corresponding graph product. We continue with Section \ref{tri}, where we characterise 0-span graphs for each variant. Moreover, we present an infinite family of graphs for which the vertex and edge variants of the corresponding span are equal. Finally, we show that for a given connected graph $H$ any span variant of $H$ can be computed in polynomial time. We conclude the paper with several open problems.

\section{Preliminary results}\label{sec:preliminary}
Our terminology and notation mostly follow \cite{West} for basic concepts of graph theory and \cite{HIK2011knjiga} for concepts related to product graphs. For any undefined terminology we refer the reader to the mentioned references, however for completeness of this paper some concepts are defined here.

Let $G$ be a connected graph and $v$ be a vertex of G. The \emph{eccentricity} of the vertex $v$, denoted $\e(v)$, is the maximum distance from $v$ to any vertex of $G$. That is, $\e(v)=\max\{d_G(v,u)\ |\ u \in V(G)\}.$ The \emph{radius} of $G$, denoted $\rad(G)$, is the minimum eccentricity among the vertices of $G$. Therefore, $\rad(G)=\min\{\e(v)\ |\ v \in V(G)\}.$ The \emph{diameter} of $G$, denoted $\diam(G)$, is the maximum eccentricity among the vertices of $G$, thus, $\diam(G)=\max\{\e(v)\ |\ v \in V(G)\}.$ 

For graphs $G$ and $H$ we will use the notation $G\subseteq H$ to denote that $G$ is a \emph{subgraph of $H$}. Moreover, $G\subseteq_C H$ denotes that $G$ is a connected subgraph of $H$. For a set $U$ of vertices of a graph $G$ we denote 
by $\left\langle U \right\rangle _G$ the \emph{subgraph of $G$ induced by the set $U$}. 

Let $H$ be a graph and let $G\subseteq H$. The graph $H-G$ is defined by $V(H-G)=V(H)\setminus V(G)$ and $E(H-G)=E(H)\setminus\{uv\in E(H) \ | \ u\in V(G) \}.$

Let $H$ be a graph. A graph $K\subseteq H$ is a \emph{component} of $H$, if $K$ is connected and for any connected graph $G\subseteq H$ it holds that if
$V(G)\cap V(K)\neq \emptyset$, then $G\subseteq K.$

In this paper, we deal with weak homomorphisms which are a generalisation of homomorphisms. For a good overview of results on homomorphisms we refer the reader to \cite{HellNese}. 

Let $G$ and $H$ be any graphs. A function $f: V(G) \rightarrow V(H)$ is a \emph{weak homomorphism from $G$ to $H$} if  for all $u,v\in V(G)$, $uv\in E(G)$ implies $f(u) f(v) \in E(H)$ or $f(u)=f(v)$. We will use the more common notation $f:G\rightarrow H$ to say that $f: V(G) \rightarrow V(H)$ is a weak homomorphism. A weak homomorphism $f: G \rightarrow H$ is \emph{surjective} if $f(V(G)) = V(H)$. A weak homomorphism $f: G \rightarrow H$ is \emph{edge surjective} if it is surjective and for every $uv \in E(H)$ there exists an edge $xy \in E(G)$ such that $u=f(x)$ and $v=f(y)$. For a weak homomorphism $f: G \rightarrow H$, the \emph{image $f(G)$ of the graph $G$} is the graph defined by $V(f(G)) = \{ f(u) \ |\ u \in V(G)\}$ and $E(f(G)) = \{ f(u)f(v) \ |\ uv\in E(G) \text{ and } f(u) \neq f(v)\}.$

Let $f: G \rightarrow H$ be a weak homomorphism from $G$ to $H$ and let $K\subseteq G$. The restriction 
$f|_{K}:K\rightarrow H$ is defined by 
$f|_{K}(u)=f(u)$ for any $u\in V(K)$.

Note that if $f: G \rightarrow H$ is a weak homomorphism, then $f(G)\subseteq H$.
The following lemma is a well-known result. Since the proof is short, we give it anyway. 
\begin{lemma}\label{lem:connectedWH}
Let $f: G \rightarrow H$ be a weak homomorphism. If $G$ is a connected graph, then also $f(G)$ is a connected graph.
\end{lemma}
\begin{proof}
Note that since $f: G \rightarrow H$ is a weak homomorphism, it follows that
\begin{displaymath}
d_{f(G)}(f(u),f(v))\leq d_G(u,v)
\end{displaymath}
for any $u,v\in V(G)$. Therefore, there is a path in $f(G)$ from $f(u)$ to $f(v)$ for any $u,v\in V(G)$. 
\end{proof}

Let $G$ and $H$ be any graphs. In the present paper, we deal with the following three products of $G$ and $H$, see \cite{HIK2011knjiga}.
The \emph{Cartesian product $G \BoxProduct H$} is defined by $V(G\BoxProduct H)=V(G)\times V(H)$ and $E(G\BoxProduct H)=\{(u_1,v_1)(u_2,v_2)  \ |\ u_1=u_2 \text{ and } v_1v_2\in E(H) \text{ or } u_1u_2\in E(G) \text{ and } v_1=v_2\}.$ 
The \emph{direct product $G\times H$} is defined by $V(G\times H)=V(G)\times V(H)$ and $E(G\times H)= \{(u_1,v_1)(u_2,v_2)  \ | \ u_1u_2\in E(G) \text{ and } v_1v_2\in E(H)\}.$ The \emph{strong product $G\boxtimes H$} is defined by $V(G\boxtimes H)=V(G)\times V(H)$ and $E(G\boxtimes H)=E(G\BoxProduct H) \cup E(G\times H).$

Let $G$ and $H$ be any graphs. The functions $
p_1: V(G)\times V(H) \rightarrow V(G) \text{ and } p_2: V(G)\times V(H) \rightarrow V(H),
$
defined by $p_1(u,v)=u$ and $p_2(u,v)=v$ for each $(u,v)\in V(G)\times V(H)$ are called the \emph{first} and the \emph{second} \emph{projection functions}, respectively. We also refer to them as the projection functions or simply, the projections.

\begin{observation}\cite{HIK2011knjiga}
Let $G$ and $H$ be any graphs. In each of the product graphs $G\times H$, $G\BoxProduct H$ and $G\boxtimes H$ both projections are weak homomorphisms.
\end{observation}

At the end of the section, we define a distance between two homomorphisms, which will be used in Section \ref{dva} to introduce all the variants of the spans of graphs.
\begin{definition}\label{def:m}
Let $f,g: G \rightarrow H$ be weak homomorphisms. We define 
\[ 
    m_G(f,g) = \min \{ d_H(f(u), g(u)) \ |\ u \in V(G) \}
\] 
to be the \emph{distance from $f$ to $g$}. 
\end{definition}
\begin{observation}\label{obs:diameter}
Let $f,g: G \rightarrow H$ be weak homomorphisms. Note that 
\begin{displaymath}
m_G(f,g)\leq \diam(H),
\end{displaymath}
if $G$ is connected. If $G$ is not connected, then $m_G(f,g)$ may equal $\infty$.
\end{observation}

\begin{lemma}\label{mLessRad}
If $f,g: G \rightarrow H$ are surjective weak homomorphisms and $G$ is connected, then
\begin{displaymath}
m_G(f,g)\leq \rad(H).
\end{displaymath}
\end{lemma}
\begin{proof}
Let $G$ be a connected graph and $f,g: G \rightarrow H$ be surjective weak homomorphisms. Let $u \in V(H)$ be such that $\e(u) = \rad(H)$, i.e. $u$ is a vertex of $H$ with eccentricity equal to the radius of $H$. Since $f$ is surjective, there is a vertex $v\in V(G)$ such that $f(v) = u$. Therefore $d_H(f(v), g(v)) \leq \rad(H)$. This implies that $m_G(f,g) \leq \rad(H)$.
\end{proof}

\begin{definition}\label{def:epsilon}
Let $H$ be a connected graph and let $Z$ be a graph such that $V(Z)\subseteq V(H) \times V(H)$. We define 
\begin{displaymath}
\varepsilon_{H}(Z) = \min\{d_H(x,y)\ |\ (x,y) \in V(Z)\}.
\end{displaymath}
\end{definition}

\section{Span - definitions and basic properties}\label{dva}
Here we introduce six different variants of a span of a given connected graph: the strong edge span, the strong vertex span, the direct edge span, the direct vertex span, the Cartesian edge span, and the Cartesian vertex span. 

All the variants model the maximum safety distance kept by two players moving through a graph, where the players may only move with accordance to a specific set of rules. Moreover, for each set of rules we define the vertex and the edge span variant. In the vertex variant of a span, all vertices of the graph must be visited at least once by both players, and in the edge variant of a span all edges must be traversed by both at least once (and therefore also all vertices).

\subsection{Strong span}
For the strong span variant the players may move with accordance to the traditional movement rules. These rules can be described using weak homomorphisms from some path to a connected graph on which the game is played, since paths can be seen as time parameters of walks.

We now define the strong edge and the strong vertex span of a graph.

\begin{definition}\label{def:strongSpans}
Let $H$ be a connected graph. 
Define 
\begin{align*}
  \seSpan{H} = \max \{ m_P(f,g) \ |\ &f,g: P \rightarrow H \text{ are edge surjective weak homomorphisms and $P$ is a path} \}.  
\end{align*}

We call $\seSpan{H}$ the \emph{strong edge span} of the graph $H$.

Define 
\begin{align*}
     \svSpan{H} = \max \{ m_P(f,g) \ |\ &f,g: P \rightarrow H \text{ are surjective weak homomorphisms and $P$ is a path} \}. 
\end{align*}
 
We call $\svSpan{H}$ the \emph{strong vertex span} of the graph $H$.
\end{definition}

\begin{observation}\label{obs:1}
Note that for any connected graph $H$, the sets
\begin{align*}
\{ m_P(f,g) \ |\ &f,g: P \rightarrow H \text{ are edge surjective weak homomorphisms and $P$ is a path} \}
\end{align*}
and
\begin{align*}
   \{ m_P(f,g) \ |\ &f,g: P \rightarrow H \text{ are surjective weak homomorphisms and $P$ is a path} \}
   \end{align*}
   are non-empty subsets of non-negative integers and are bounded from above by $\rad(H)$. Therefore,  $\svSpan{H}$ and $\seSpan{H}$ are well-defined.
   Note also that
\begin{displaymath}
\seSpan{H}\leq \svSpan{H}\leq \rad(H).
\end{displaymath}
\end{observation}

Note that in the above definition, paths may be replaced by any connected graph, as seen in the following proposition. This proves that the concept of spans of a graph is an application of the notion of spans of continua from \cite{Lelek} where spans are defined by going through all continua, i.e., compact connected metric spaces $G$, and not just through all arcs (paths $P$). In the theory of continua these two are not equivalent, since there are continua that are not path connected.

\begin{proposition}\label{prop:strongSpans}
Let $H$ be a connected graph. 
Then
\begin{align*}
     \svSpan{H} = \max \{ m_G(f,g) \ |\ &f,g: G \rightarrow H \text{ are surjective weak homomorphisms} \\ &\text{and $G$ is connected} \}
\end{align*}
and
\begin{align*}
  \seSpan{H} = \max \{ m_G(f,g) \ |\ &f,g: G \rightarrow H \text{ are edge surjective weak homomorphisms} \\ &\text{and $G$ is connected}  \}.
\end{align*}
\end{proposition}

\begin{proof}
    Let \begin{align*}
     A = \{ m_P(f,g) \ |\ &f,g: P \rightarrow H \text{ are surjective weak homomorphisms and $P$ is a path} \} 
\end{align*}
and 
\begin{align*}
     B = \{ m_G(f,g) \ |\ &f,g: G \rightarrow H \text{ are surjective weak homomorphisms and $G$ is connected} \}. 
\end{align*}

Since $A \subseteq B$, it follows that $\max A \leq \max B$. To show that $\max B \leq \max A$, let $G$ be any connected graph and $f,g: G \rightarrow H$ be any surjective weak homomorphisms. We show that there is a path $P$ and surjective weak homomorphisms $f', g': P \rightarrow H$ such that $m_G(f,g) = m_P(f',g')$. Let $W=(w_0, w_1, \ldots, w_k)$, where $w_i w_{i+1} \in E(G)$ for each $i \in \{0, 1, \ldots, k-1\}$, be any walk through all vertices of $G$, let $P$ be a path with the vertex set $\{p_0, p_1, \ldots, p_k\}$, where $p_i p_{i+1} \in E(P)$ for each $i \in \{0, 1, \ldots, k-1\}$, and let $h: P \rightarrow G$ be defined by $h(p_i)=w_i$ for each $i \in \{0, 1, \ldots, k\}$. Note that $f\circ h$ and $g\circ h$ are surjective weak homomorphisms from $P$ to $H$ such that $m_P(f\circ h, g\circ h) = m_G(f, g)$. This proves that \begin{align*}
  \svSpan{H} = \max \{ m_G(f,g) \ |\ &f,g: G \rightarrow H \text{ are surjective weak homomorphisms} \\ &\text{and $G$ is connected} \}
\end{align*}

The proof of \begin{align*}
  \seSpan{H} = \max \{ m_G(f,g) \ |\ &f,g: G \rightarrow H \text{ are edge surjective weak homomorphisms} \\ &\text{and $G$ is connected} \}
\end{align*} is analogous to the proof above, with additional assumption that the walk $W$ is a walk through all edges of the graph.
\end{proof}

For any connected graph $H$, the following two theorems show that it suffices to consider only connected subgraphs $Z$ of $H\boxtimes H$ and the projections $p_1, p_2: H\boxtimes H \rightarrow H$ when determining the corresponding span of a graph. In particular, $Z$ can be viewed as the subgraph of $H\boxtimes H$ induced by all vertices $(u, v)$ such that at some point during the game on $H$, Alice occupies $u$ while Bob occupies $v$. Thus, $p_1(Z)$ (resp. $p_2(Z)$) represents the subgraph of $H$ traversed by Alice (resp. Bob) throughout the duration of the game, while $\varepsilon_{H}(Z)$ represents the minimum distance between Alice and Bob. 

\begin{theorem}\label{thm:s-vertexSpanOnStrongProduct}
If $H$ is a connected graph, then 
\begin{equation*}
    \svSpan{H} = \max \{ \varepsilon_{H}(Z) \ |\ Z \subseteq_C H\boxtimes H \text{ with } p_1(V(Z))=p_2(V(Z))=V(H) \}.
\end{equation*}
\end{theorem}

\begin{proof}
We define
\begin{align*}
    A &= \{ \varepsilon_{H}(Z) \ |\ Z \subseteq_C H\boxtimes H \text{ with } p_1(V(Z))=p_2(V(Z))=V(H) \} \text{ and } \\
    B &= \{ m_G(f,g) \ |\ f,g: G \rightarrow H \text{ are surjective weak homomorphisms and $G$ is connected} \}.
\end{align*}
We will show that $\max(A) = \max (B)$ by proving that $A=B$. First we show that $A\subseteq B$. Let $r \in A$ be arbitrary. Let $Z$ be a connected subgraph of $H \boxtimes H$ such that $p_1(V(Z))=p_2(V(Z))=V(H)$ and $\varepsilon_{H}(Z) = r.$  Let $G=Z, f=p_1|_{G}$ and $g=p_2|_{G}$. Then 
\begin{enumerate}
    \item $G$ is a connected graph, 
    \item $f, g: G \rightarrow H$ are surjective weak homomorphisms, and
    \item \begin{align*}
    m_G(f,g) &= \min \{d_H(f(u), g(u))\ |\ u \in V(G) \} \\
           &= \min \{d_H(p_1(u), p_2(u))\ |\ u \in V(Z) \} \\
           &= \min \{d_H(x,y) \ |\ (x,y)\in V(Z) \} \\
           &= \varepsilon_{H}(Z) = r.
\end{align*}
\end{enumerate}
Therefore, $r \in B$ and we have proved that $A \subseteq B$.  

To show that $B\subseteq A$, let $r\in B$ be arbitrary. Let $G$ be a connected graph and let $f,g: G \rightarrow H $ be surjective weak homomorphisms such that $m_G(f,g)=r$. Define $\psi:V(G)\rightarrow V(H\boxtimes H)$ by $\psi(u) = (f(u), g(u))$ for all $u\in V(G)$. We claim that $\psi$ is a well-defined weak homomorphism. It is obvious that $\psi(u)\in V(H\boxtimes H)$ for any $u\in V(G)$. Let $uv\in E(G)$. The following cases are possible:
\begin{enumerate}
    \item $f(u)=f(v)$ and $g(u)=g(v)$. Here $\psi(u)=\psi(v)$.
    \item $f(u)f(v)\in E(H)$ and $g(u)=g(v)$. Here $\psi(u)=(f(u), g(u))=(f(u), g(v))$ and $\psi(v)=(f(v), g(v))$. Therefore $\psi(u)\psi(v)\in E(H\boxtimes H)$. 
    \item $f(u)=f(v)$ and $g(u)g(v)\in E(H)$. Here $\psi(u)=(f(u), g(u))=(f(v), g(u))$ and $\psi(v)=(f(v), g(v))$. Therefore $\psi(u)\psi(v)\in E(H\boxtimes H)$. 
    \item $f(u)f(v)\in E(H)$ and $g(u)g(v)\in E(H)$. Here $\psi(u)=(f(u), g(u))$ and $\psi(v)=(f(v), g(v))$. Therefore $\psi(u)\psi(v)\in E(H\boxtimes H)$. 
\end{enumerate}
It follows that $\psi$ is a well-defined weak homomorphism from $G$ to $H\boxtimes H$. 
Let $Z=\psi (G)$. Since $G$ is connected, it follows that $Z$ is a connected subgraph of $H\boxtimes H$. Next we show that $p_1(V(Z))=V(H)$ and $p_2(V(Z))=V(H)$. Let $x\in V(H)$. Since $f$ and $g$ are surjective weak homomorphisms, there are $u,v\in V(G)$ such that $f(u)=x$ and $g(v)=x$. Then $p_1(f(u),g(u))=x$ and $p_2(f(v),g(v))=x$ and we are done. 

Since
\begin{align*}
    \varepsilon_{H}(Z) &= \min\{ d_H(x,y)\ |\ (x,y) \in V(Z) \} \\
                   &= \min\{d_H(f(u), g(u)) \ |\ u \in V(G) \} \\
                   &= m_G(f,g) = r,
\end{align*}
it follows that $r \in A$. Hence, $B \subseteq A$ and we have proved that $A=B$. It follows that $\max(A) = \max(B)$. Using Proposition \ref{prop:strongSpans} the assertion follows immediately.
\end{proof}

\begin{theorem}\label{thm:s-edgeSpanOnStrongProduct}
Let $H$ be a connected graph. Then 
\begin{equation*}
    \seSpan{H} = \max \{ \varepsilon_{H}(Z) \ |\ Z \subseteq_C H\boxtimes H \text{ with } p_1(Z)=p_2(Z)=H \}.
\end{equation*}
\end{theorem}

\begin{proof}
We define
\begin{align*}
    A &= \{ \varepsilon_{H}(Z) \ |\ Z \subseteq_C H\boxtimes H \text{ with } p_1(Z)=p_2(Z)=H \} \text{ and }\\
    B &= \{ m_G(f,g) \ |\ f,g: G \rightarrow H \text{ are edge surjective weak homomorphisms and $G$ is connected} \}.
\end{align*}
Similarly to the proof of Theorem \ref{thm:s-vertexSpanOnStrongProduct} we prove the assertion by proving that $A = B$. To show that $A\subseteq B$, let $r \in A$ be arbitrary. Let $Z \subseteq_C H \boxtimes H$ be such that $p_1(Z)=p_2(Z)=H$ and $\varepsilon_{H}(Z) =  r$.  Let $G=Z, f=p_1|_{G}$ and $g=p_2|_{G}$. Then 
\begin{enumerate}
    \item $G$ is a connected graph, 
    \item $f, g: G \rightarrow H$ are edge surjective weak homomorphisms, and
    \item \begin{align*}
    m_G(f,g) &= \min \{d_H(f(u), g(u))\ |\ u \in V(G) \} \\
           &= \min \{d_H(p_1(u), p_2(u))\ |\ u \in V(Z) \} \\
           &= \min \{d_H(x,y) \ |\ (x,y)\in V(Z) \} \\
           &= \varepsilon_{H}(Z) = r.
\end{align*}
\end{enumerate}

Therefore, $r \in B$ and we have proved that $A \subseteq B$.  
To show that $B\subseteq A$, let $r\in B$. Let $G$ be a connected graph and let $f,g: G \rightarrow H$ be edge surjective weak homomorphisms such that $m_G(f,g)=r$. Let $Z$ be the graph defined by
\[
  V(Z) = \{ (f(u), g(u))\ |\ u \in V(G) \}
\]
and, for any two vertices $(u,v)$ and $(u', v')$ of the graph  $Z$, $(u,v)(u', v')\in E(Z)$ if and only if one of the following three conditions is satisfied:
\begin{enumerate}
\item $u u' \in E(H)$ and $v=v'$, or 
\item $v v' \in E(H)$ and $u=u'$, or 
\item $u u' \in E(H)$ and $v v' \in E(H)$.
\end{enumerate}
Define $\psi:V(G)\rightarrow V(Z)$ by $\psi(u) = (f(u), g(u))$ for all $u\in V(G)$. Note that $\psi$ is an edge surjective weak homomorphism from $G$ to $Z$. Therefore, $Z=\psi(G)$ and since $G$ is connected, by Lemma \ref{lem:connectedWH} the graph $Z$ is also connected. Therefore, $Z \subseteq_C H\boxtimes H$. Since
\begin{align*}
    \varepsilon_{H}(Z) &= \min\{ d_H(x,y)\ |\ (x,y) \in V(Z) \} \\
                   &= \min\{d_H(f(u), g(u)) \ |\ u \in V(G) \} \\
                   &= m_G(f,g) = r,
\end{align*}
it follows that $r \in A$. Hence, $B\subseteq A$ and we have proved that $A=B$. It follows that $\max(A) = \max(B)$. Using Proposition \ref{prop:strongSpans} the assertion follows immediately.
\end{proof}

\subsection{Direct span}
For the direct span variant the players may only move according to the active movement rules. This rule can be described using aligned weak homomorphisms.

We now define the direct edge and the direct vertex span of a graph.

\begin{definition}
Let $f,g:G\rightarrow H$ be weak homomorphisms. We say that $f$ and $g$ are aligned weak homomorphisms, if for any $uv\in E(G)$,
\begin{displaymath}
f(u)f(v)\in E(H)    \Longleftrightarrow     g(u)g(v)\in E(H).
\end{displaymath}
\end{definition}
\begin{definition}\label{def:directSpans}
Let $H$ be a connected graph. 
Define 
\begin{align*}
  \deSpan{H} = \max \{ m_P(f,g) \ |\ &f,g: P \rightarrow H \text{ are edge surjective aligned weak homomorphisms} \\ &\text{and $P$ is a path} \}.  
\end{align*}

We call $\deSpan{H}$ the \emph{direct edge span} of the graph $H$.

Define 
\begin{align*}
     \dvSpan{H} = \max \{ m_P(f,g) \ |\ &f,g: P \rightarrow H \text{ are surjective aligned weak homomorphisms} \\ &\text{and $P$ is a path} \}. 
\end{align*}
 
We call $\dvSpan{H}$ the \emph{direct vertex span} of the graph $H$.
\end{definition}

\begin{observation}\label{obs:2}
Note that for any connected graph $H$, 
\begin{displaymath}
\deSpan{H}\leq \dvSpan{H}\leq \rad(H).
\end{displaymath}
\end{observation}

\begin{observation}\label{obs:pathtoconnDirect}
Note that, 
\begin{align*}
  \deSpan{H} = \max \{ m_G(f,g) \ |\ &f,g: G \rightarrow H \text{ are edge surjective aligned weak homomorphisms} \\ &\text{and $G$ is connected} \}
\end{align*} and 
\begin{align*}
     \dvSpan{H} = \max \{ m_G(f,g) \ |\ &f,g: G \rightarrow H \text{ are surjective aligned weak homomorphisms} \\ &\text{and $G$ is connected} \} 
\end{align*}
can be proved similarly as Proposition \ref{prop:strongSpans}.
\end{observation}

For any connected graph $H$, the following theorem shows that it is not necessary to consider all corresponding weak homomorphisms from all possible connected graphs $G$. Instead it suffices to consider only connected subgraphs of $H\times H$ and the projections $p_1, p_2: H\times H \rightarrow H$.

\begin{theorem}\label{thm:ev-spanOnDirectProduct}
If $H$ is a connected graph, then 
\[
    \dvSpan{H} = \max \{ \varepsilon_{H}(Z) \ |\ Z \subseteq_C H\times H \text{ with } p_1(V(Z))=p_2(V(Z))=V(H) \}
\]
and
\[
    \deSpan{H} = \max \{ \varepsilon_{H}(Z) \ |\ Z \subseteq_C H\times H \text{ with } p_1(Z)=p_2(Z)=H \}.
\]
\end{theorem}

\begin{proof}
The proof follows the same line of thought as the proofs of Theorems \ref{thm:s-vertexSpanOnStrongProduct} and \ref{thm:s-edgeSpanOnStrongProduct}. The only differences are in the construction of the corresponding connected subgraphs of (in this case) $H \times H$. 
\end{proof}

\subsection{Cartesian span}
For the Cartesian span variant the players may only move according to the lazy movement rules. These rules can be described using opposite weak homomorphisms.

We now define the Cartesian edge and the Cartesian vertex span of a graph.

\begin{definition}
Let $f,g:G\rightarrow H$ be weak homomorphisms. We say that $f$ and $g$ are opposite weak homomorphisms, if for any $uv\in E(G)$,
\begin{displaymath}
f(u)f(v)\in E(H)    \Leftrightarrow     g(u)=g(v).
\end{displaymath}
\end{definition}

\begin{definition}\label{def:cartesianSpans}
Let $H$ be a connected graph. 
Define 
\begin{align*}
  \ceSpan{H} = \max \{ m_P(f,g) \ |\ &f,g: P \rightarrow H \text{ are edge surjective opposite weak homomorphisms} \\ &\text{and $P$ is a path} \}.  
\end{align*}
     
We call $\ceSpan{H}$ the \emph{Cartesian edge span} of the graph $H$.

Define 
\begin{align*}
     \cvSpan{H} = \max \{ m_P(f,g) \ |\ &f,g: P \rightarrow H \text{ are surjective opposite weak homomorphisms} \\ &\text{and $P$ is a path} \}. 
\end{align*}
 
We call $\cvSpan{H}$ the \emph{Cartesian vertex span} of the graph $H$.
\end{definition}
\begin{observation}\label{obs:3}
Note that for any connected graph $H$, 
\begin{displaymath}
\ceSpan{H}\leq \cvSpan{H}\leq \rad(H).
\end{displaymath}
\end{observation}

\begin{observation}\label{obs:pathtoconnCart}
Note that, 
\begin{align*}
  \ceSpan{H} = \max \{ m_G(f,g) \ |\ &f,g: G \rightarrow H \text{ are edge surjective opposite weak homomorphisms} \\ &\text{and $G$ is connected} \} 
\end{align*} and 
\begin{align*}
     \cvSpan{H} = \max \{ m_G(f,g) \ |\ &f,g: G \rightarrow H \text{ are surjective opposite weak homomorphisms} \\ &\text{and $G$ is connected} \}
\end{align*}
can be proved similarly as Proposition \ref{prop:strongSpans}.
\end{observation}

For any connected graph $H$, the following theorem shows that it is not necessary to consider all corresponding weak homomorphisms from all possible connected graphs $G$. Instead it suffices to consider only connected subgraphs of $H\BoxProduct H$ and the projections $p_1, p_2: H\BoxProduct H \rightarrow H$.

\begin{theorem}\label{thm:s-spanOnCartesianProduct}
If $H$ is a connected graph, then 
\[
    \cvSpan{H} = \max \{ \varepsilon_{H}(Z) \ |\ Z \subseteq_C H\BoxProduct H \text{ with } p_1(V(Z))=p_2(V(Z))=V(H) \}
\]
and
\[
    \ceSpan{G} = \max \{ \varepsilon_{H}(Z) \ |\ Z \subseteq_C H\BoxProduct H \text{ with } p_1(Z)=p_2(Z)=H \},
\]
\end{theorem}

\begin{proof}
Again, the proof follows the same line of thought as the proofs of Theorems \ref{thm:s-vertexSpanOnStrongProduct} and \ref{thm:s-edgeSpanOnStrongProduct}. The only differences are in the construction of the corresponding connected subgraphs of (in this case) $H \BoxProduct H$.
\end{proof}

\section{0-span graphs and graphs with equal vertex and edge span variant}\label{tri}

In this section we focus on 0-span graphs; i.e., graphs in which it is impossible to keep a positive safety distance at all points in time in any of the above introduced models. We also construct an infinite family of graphs for which the corresponding vertex and edge span variants are equal. 

First we give and prove the following characterisations of graphs with strong vertex span, strong edge span, direct vertex span and direct edge span equal to 0. In the first result that follows, we prove the assertion formally, by constructing a corresponding subgraph with (edge) surjective projections and then present the same result in the language of Alice and Bob moving on a graph. For the later results we only prove them in the setting of players in the graph.

\begin{theorem}\label{thm:ssZero}
Let $H$ be any connected graph. The following statements are equivalent.
\begin{enumerate}
    \item $\seSpan{H}=0$.
    \item $\svSpan{H}=0$.
    \item $|V(H)|=1$. 
\end{enumerate}
\end{theorem}
\begin{proof}
Let $|V(H)|=1$. Then $V(H)=\{u\}$ for some $u$ and, obviously, $\svSpan{H}=\seSpan{H}=0$. 
Next, let $\seSpan{H}=0$ or $\svSpan{H}=0$. We show that $|V(H)|=1$. Suppose that $|V(H)|>1$. Let $u,v\in V(H)$ such that $uv\in E(H)$. 

If $V(H)=\{u,v\}$, then let $Z$ be defined as follows. Let
\begin{math}
V(Z)=\{(u,v),(v,u)\}
\end{math}
and
\begin{math}
E(Z)=\{(u,v)(v,u)\}.
\end{math}
Then $\varepsilon_{H}(Z)=1$ and, therefore, $\svSpan{H}>0$ and $\seSpan{H}>0$. 

Let $|V(H)|>2$, $u_0 v_0 \in E(H)$ and let $G$ be the graph defined by $V(G)=\{u_0,v_0\}$ and $E(G)=\{u_0 v_0\}$. Also, let $H_1$, $H_2$, $\ldots$, $H_m$ be the components of $H-G$. Since $V(H)\neq \{u_0, v_0\}$, it follows that $m>0$. For each $i\in \{1, 2, \ldots, m\}$, let 
\begin{displaymath}
U_i = N(u_0) \cap V(H_i)
\end{displaymath} 
and 
\begin{displaymath}
V_i = N(v_0) \cap V(H_i).
\end{displaymath}
We define a graph $Z$ as follows. Let
\begin{displaymath}
V(Z)=\left(\bigcup_{i=1}^{m}V(H_i\boxtimes G)\right) \cup \left(\bigcup_{i=1}^{m}V(G\boxtimes H_i)\right)\cup \{(u_0,v_0),(v_0,u_0)\}
\end{displaymath}
and
\begin{align*}
E(Z)=&\left(\bigcup_{i=1}^{m}E(H_i\boxtimes G)\right)  \cup \left(\bigcup_{i=1}^{m}E(G\boxtimes H_i)\right)\cup \{(u_0,v_0)(v_0,u_0)\} \cup\\ &
\bigcup_{i=1}^m \left(\left(\bigcup_{u\in U_i} \{(u,v_0)(u_0,v_0)\} \right) \cup \left(\bigcup_{v \in V_i} \{(v,u_0)(v_0,u_0)\} \right)\right) \cup \\ &
\bigcup_{i=1}^m \left(\left(\bigcup_{u\in U_i} \{(v_0,u)(v_0,u_0)\} \right) \cup \left(\bigcup_{v \in V_i} \{(u_0,v)(u_0,v_0)\} \right)\right).
\end{align*}

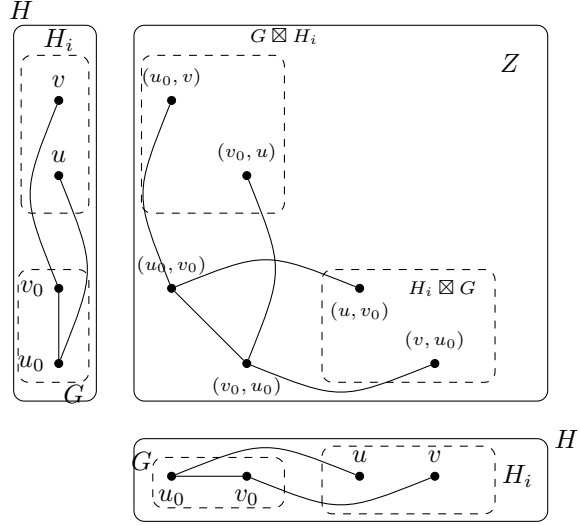
\begin{figure}[!ht]
	\centering
\begin{tikzpicture}[scale=1]
\tikzstyle{rn}=[circle, fill=black,draw, inner sep=0pt, minimum size=3pt]

\draw[rounded corners] (0, 0) rectangle (5.5, 5);
\node (Z) at (5, 4.5) {$Z$};
\draw[rounded corners] (-1.6, 0) rectangle (-0.5, 5);
\node (Hh) at (5.75, -0.5) {$H$};
\draw[rounded corners] (0, -0.5) rectangle (5.5, -1.6);
\node (Hv) at (-1.5, 5.2) {$H$};

\draw[rounded corners, dashed] (-1.5, 2.5) rectangle (-0.6, 4.6);
\node (Hvi) at (-1, 4.8) {$H_i$};
\node at (0.1,-0.8) {$G$};

\draw[rounded corners, dashed] (2.5, -1.5) rectangle (4.8, -0.6);
\node (Hhi) at (5.1, -1) {$H_i$};
\node at (-0.8,0.1) {$G$};

\draw[rounded corners, dashed] (-1.54, 0.25) rectangle (-0.6, 1.75);
\node [style=rn] (uh0) at (0.5, -1) [label=below:$u_0$] {}; 
\node [style=rn] (vh0) at (1.5, -1) [label=below:$v_0$] {}; 
\draw (uh0)--(vh0);

\draw[rounded corners, dashed] (0.25, -1.42) rectangle (2, -0.75);
\node [style=rn] (uv0) at (-1, 0.5) [label=left:$u_0$] {}; 
\node [style=rn] (vv0) at (-1, 1.5) [label=left:$v_0$] {}; 
\draw (uv0)--(vv0);

\node [style=rn] (u0v0) at (0.5, 1.5) [label=above:{\scriptsize $(u_0,v_0)$}] {};
\node [style=rn] (v0u0) at (1.5, 0.5) [label=below:{\scriptsize $(v_0,u_0)$}] {};
\draw (u0v0)--(v0u0);

\draw[rounded corners, dashed] (2.5, 0.25) rectangle (4.8, 1.75);
\node at (4.8, 1.5) [label=left:{\scriptsize $H_i \boxtimes G$}] {};

\draw[rounded corners, dashed] (0.1, 2.5) rectangle (2, 4.6);
\node at (2, 4.5) [label=above:{\scriptsize $G \boxtimes H_i$}] {};

\node [style=rn] (uh) at (3, -1) [label=above:{$u$}] {};
\node [style=rn] (uv) at (-1, 3) [label=above:{$u$}] {};
\draw (uh0) .. controls (1.75, -0.5) .. (uh);
\draw (uv0) .. controls (-0.5, 1.75) .. (uv);

\node [style=rn] (uv0) at (3, 1.5) [label=below:{\scriptsize $(u,v_0)$}] {};
\draw (uv0) .. controls (1.75, 2) .. (u0v0);

\node [style=rn] (v0u) at (1.5, 3) [label=above:{\scriptsize $(v_0, u)$}] {};
\node [style=rn] (u0v) at (0.5, 4) [label=above:{\scriptsize $(u_0, v)$}] {};
\draw (v0u) .. controls (2, 1.75) .. (v0u0);
\draw (u0v) .. controls (0, 2.75) .. (u0v0);

\node [style=rn] (vh) at (4, -1) [label=above:{$v$}] {};
\node [style=rn] (vv) at (-1, 4) [label=above:{$v$}] {};
\draw (vh0) .. controls (2.75, -1.5) .. (vh);
\draw (vv0) .. controls (-1.5, 2.75) .. (vv);

\node [style=rn] (vu0) at (4, 0.5) [label=above:{\scriptsize $(v,u_0)$}] {};
\draw (vu0) .. controls (2.75, 0) .. (v0u0);
        
\end{tikzpicture}
\caption{Sketch of the construction of the graph $Z$ in the case of strong spans.}
		\label{fig:sketch1}
\end{figure}

Figure \ref{fig:sketch1} shows an example of how the vertices and edges are added to $Z$ for any component $H_i$.
It is clear that $Z$ is connected and that $p_1(Z)=p_2(Z)=H$ (and therefore also $p_1(V(Z))=p_2(V(Z))=V(H)$). Since for any vertex $u \in V(H)$ it holds true that $(u,u) \not \in V(Z)$, therefore $\varepsilon_{H}(Z)>0$. It follows that $\svSpan{H}>0$ and $\seSpan{H}>0$.
\end{proof}

The part from the proof of Theorem \ref{thm:ssZero}, where $|V(H)|>2$, can be stated in the setting of players on the graph, as follows. Remember, the traditional movement rules apply. We will show that Alice and Bob can visit all vertices and edges without ever being on the same vertex by providing their movements through $H$, and therefore neither $\seSpan{H}=0$ nor $\svSpan{H}=0$. First, Alice and Bob start on different vertices of the graph, say the vertex $u$ is where Alice starts, and $v$ is Bob's beginning location. Alice can visit all vertices of $H$ in a breadth-first search (BFS) order with respect to the starting vertex $u$. For each vertex $w$ in the BFS order, Alice first moves from $u$ to $w$. Then for each neighbour $x$ of the vertex $w$, she moves to $x$ and back to $w$. Finally, she returns to $u$. In this manner she visits all edges of the graph and therefore also all vertices. If at any point in time Alice has to move to Bob's vertex, they simply swap vertices by moving along the same edge at the same time, otherwise Bob does not move while Alice is moving. After Alice has finished moving, Alice and Bob swap roles and Bob executes the same procedure, while Alice remains at the same vertex, unless she swaps positions with Bob along the same edge. Clearly, at any point in time, Alice and Bob are at distance at least 1 and therefore $\svSpan{H}>0$ and $\seSpan{H}>0$.

\begin{theorem}\label{equal:s}
If $H$ is the one-vertex graph or $\rad(H)=1$, then 
\begin{displaymath}
\seSpan{H}=\svSpan{H}.
\end{displaymath}
\end{theorem}
\begin{proof}
The case when $H$ is the one-vertex graph follows directly from Theorem \ref{thm:ssZero}. Now let $H$ be such that $\rad(H) = 1$. Also from Theorem \ref{thm:ssZero} it follows that $\svSpan{H}\neq 0$ and $\seSpan{H}\not = 0$. Since $\rad(H)=1$, using Observation \ref{obs:1} we immediately obtain that $\svSpan{H}=\seSpan{H}=1$. 
\end{proof}

Note that for a path $P_n$, for any integer $n$, it holds that $\svSpan{P_n}=\seSpan{P_n}$. Moreover, for any $n>1$, $\svSpan{P_n}=\seSpan{P_n} = 1$. Also, for any $n>3$, $\rad(P_n)>1$.  Therefore there are graphs $H$ such that $\rad(H)>1$ and $\svSpan{H} = \seSpan{H}$. Hence, we present the following open problem.

\begin{problem}
Find all connected graphs $H$ for which $\seSpan{H} = \svSpan{H}$.
\end{problem}

Note, for any tree $T$ visiting all vertices requires visiting all edges, therefore $\seSpan{T}=\svSpan{T}$.

\begin{theorem}\label{thm:dsZero}
Let $H$ be any connected graph. The following statements are equivalent.
\begin{enumerate}
    \item $\deSpan{H}=0$.
    \item $\dvSpan{H}=0$.
    \item $|V(H)|=1$. 
\end{enumerate}
\end{theorem}
\begin{proof}
The proof is similar to the proof of Theorem \ref{thm:ssZero} with the only difference being in the case $|V(H)|>2$. For this case we provide the proof in the language of players' movements through the graph and show that they can visit all vertices and edges without ever being on the same vertex, thus neither $\deSpan{H}=0$ nor $\dvSpan{H}=0$.

Alice and Bob can move to visit every vertex and edge of the graph in a similar manner as described after the proof of Theorem \ref{thm:ssZero} with a small change. This time, Alice and Bob start on different end-vertices of the same edge, say $e=uv$, moreover assume $u$ is where Alice starts, and $v$ is where Bob starts. During every move Alice makes, her movements are the same as described in the mentioned case above, Bob must also move, since active movement rules apply. Instead of remaining at the same vertex at each move Alice makes, Bob alternates between the vertices $v$ and $u$. If at any point in time Alice and Bob both want to move to the same vertex ($u$ or $v$), they can avoid such moves by retracing all their steps to the starting positions and then swapping vertices. After Alice has visited all edges and vertices, they retrace their steps to the starting position and their roles are again reversed. Bob can now visit all vertices and edges in a similar fashion without meeting Alice at the same vertex. This implies that in the case where $|V(H)>2|$ neither $\deSpan{H}=0$ nor $\dvSpan{H}=0$.
\end{proof}

\begin{theorem}\label{equal:d}
If $H$ is the one-vertex graph or $\rad(H)=1$, then 
\begin{displaymath}
\deSpan{H}=\dvSpan{H}. 
\end{displaymath}
\end{theorem}
\begin{proof}
The case when $H$ is the one-vertex graph follows directly from Theorem \ref{thm:dsZero}. Now let $H$ be such that $\rad(H) = 1$. Again, from Theorem \ref{thm:dsZero} it follows that $\dvSpan{H}\neq 0$ and $\deSpan{H}\not = 0$. Since $\rad(H)=1$, using Observation \ref{obs:2} we immediately obtain that $\dvSpan{H}=\deSpan{H}=1$. 
\end{proof}

\begin{problem}
Find all connected graphs $H$ for which $\deSpan{H} = \dvSpan{H}$.
\end{problem}

Again, since for any tree $T$ visiting all vertices requires visiting all edges, the equality above holds true for trees.

Next we give and prove the following characterisation of graphs with the Cartesian vertex span and Cartesian edge span equal to 0.

\begin{theorem}\label{thm:cSpanZero}
Let $H$ be any connected graph. The following statements are equivalent.
\begin{enumerate}
    \item $\ceSpan{H}=0$.
    \item $\cvSpan{H}=0$.
    \item There is a positive integer $n$ such that $H$ is an $n$-path.
\end{enumerate}
\end{theorem}
\begin{proof}
Let $H=P_n$ be a path on $n$ vertices, for an arbitrary positive integer $n$. Denote the vertices of $P_n$ by $v_0, v_1, \ldots, v_{n-1}$, such that for any $i\in\{0,\ldots,n-2\}$ the vertices $v_i$ and $v_{i+1}$ are adjacent. We show that $\cvSpan{H} = 0$. Using Theorem \ref{thm:s-spanOnCartesianProduct}, let $Z \subseteq_C H \BoxProduct H$ be such that $p_1(V(Z))=p_2(V(Z))=V(H)$ and $\varepsilon_H(Z) = \cvSpan{H}$. Note, the set $\{ (v_i, v_i) \ |\ i \in \{0,1,\ldots, n-1\}\}$ is a cut set of $P_n \BoxProduct P_n$ which divides the graph $P_n \BoxProduct P_n$ into two connected components, one with the vertex set $V_< = \{ (v_i, v_j) \ |\ i<j \text{ and } i,j \in \{0,1,\ldots, n-1\}\}$ and the other with the vertex set $V_> = \{ (v_i, v_j) \ |\ i>j \text{ and } i,j \in \{0,1,\ldots, n-1\}\}$. Towards contradiction suppose $\cvSpan{H} > 0$. This implies that for any $i \in \{0,1,\ldots, n-1\}$ the vertex $(v_i, v_i)$ does not belong to $V(Z)$. Since $p_1(V(Z))=V(H)$ there exists a vertex $(v_0,v_j)\in V(Z)$ with $j > 0$. Moreover, $(v_0, v_j)$ belongs to the set $V_<$. Similarly, since $p_1(V(Z))=V(H)$ there also exists a vertex $(v_{n-1}, v_l) \in V(Z)$ with $l<n-1$ and this vertex belongs to $V_>$. But then $(v_0, v_j)$ and $(v_{n-1}, v_l)$ belong to two distinct connected components, a contradiction to the fact that $Z$ is connected. Therefore, $\cvSpan{H} = 0$ and using Observation \ref{obs:3} also $\ceSpan{H} = 0$.

Next, we show that $\cvSpan{H} = 0$ or $\ceSpan{H} = 0$ implies that there exists a positive integer $n$ such that $H$ is an $n$-path. Suppose that $H$ is not an $n$-path for any positive integer $n$. This means that $H$ contains a cycle or it is a tree that contains a vertex of degree at least 3. For these two cases we show that Alice and Bob can visit all vertices and edges without ever being in the same vertex at the same time and therefore $\cvSpan{H} > 0$ and $\ceSpan{H} > 0$.

Suppose $H$ contains a cycle $C$. Let Alice and Bob start at two different vertices of the cycle $C$, say $u$ and $v$, respectively. Since lazy movement rules apply, Alice's movements can be similar to the ones described after the proof of Theorem \ref{thm:ssZero}, while Bob remains still at all times, unless Alice needs to move to Bob's current position. Since they cannot both move at the same time, Bob avoids Alice as follows. Alice does not move, Bob moves to another adjacent vertex of the the cycle $C$, stays there, and from there Alice can move to her desired vertex. After Alice has visited all vertices and edges, she can retrace her steps to move to the starting position, and so can Bob. After that Bob executes his moves using the same procedure as Alice before, and Alice can avoid using Bob's previous strategy. Since they were never at the same vertex at the same time, $\cvSpan{H} > 0$ and $\ceSpan{H} > 0$.

Finally, let $H$ be a tree that contains a vertex of degree at least three, say $u_0$, and let $u_1, u_2, u_3$ be three distinct neighbours of $u_0$. Note, since $H$ is a tree, the vertices $u_1, u_2$ and $u_3$ induce a graph with no edges. Let Alice start in the vertex $u_1$, and Bob in the vertex $u_0$. While Alice visits all vertices (and therefore edges, since $H$ is a tree) of $H-u_0$ and returns to $u_1$, Bob stays in $u_0$. Now, Alice and Bob can swap positions and still obey the lazy movement rules as follows. Bob moves to $u_2$ while Alice stays at $u_1$, then Bob stays at $u_2$ and Alice moves to $u_0$ and after that to $u_3$. Now Alice stays at $u_3$, now Bob can move to $u_0$ and after that to $u_1$. After this Alice can move to $u_0$, visit the remaining vertices and edges (while Bob does not move) and then return to $u_0$. Now the roles of Alice and Bob are exchanged, where they can apply the same procedure and Bob can visit all vertices and edges of $H$.  Again, Alice and Bob were never at the same vertex at the same time, therefore $\cvSpan{H} > 0$ and $\ceSpan{H} > 0$. This concludes the proof.
\end{proof}

\begin{theorem}\label{equal:c}
Let $H$ be a connected graph. If $H$ is a path or $\rad(H)=1$, then 
\begin{displaymath}
\ceSpan{H}=\cvSpan{H}. 
\end{displaymath}
\end{theorem}
\begin{proof}
The case where $H$ is a path follows directly from Theorem \ref{thm:cSpanZero}. Assume that $H$ is not a path and $\rad(H)=1$. From Theorem \ref{thm:cSpanZero} it follows that $\cvSpan{H}\neq 0$ and $\ceSpan{H}\not = 0$. Since $\rad(H)=1$, using Observation \ref{obs:3} we immediately obtain that $\dvSpan{H}=\deSpan{H}=1$. 
\end{proof}

\begin{problem}
Find all connected graphs $H$ for which $\ceSpan{H} = \cvSpan{H}$.
\end{problem}

As well as in the strong and direct span variants, this equality holds true for any tree.

\section{Algorithm for computing the maximal safety distance}

In this section we show that regardless of the type of movement rules, the maximal safety distance two players can keep at all times can be determined in polynomial time. Note, the algorithms are written in general and we assume that the goal - whether the players must visit all vertices or all edges of the given graph - is known. Therefore in appropriate places only one of the two possible conditions needs to be checked.

\begin{algorithm}[!ht]
\SetKwData{True}{true}
\SetKwData{False}{false}
\KwIn{graph $H$, required distance $D$, movement rules $R$}
\KwOut{\True if two players can keep the maximal distance at least $D$ while traversing all vertices/edges of $H$ under movement rules $R$, and \False, otherwise}
\DontPrintSemicolon
\BlankLine
\tcc{create the corresponding product}
\eIf{$R$ is traditional movement rules}{ \label{alg51-l1}
    $G = H \boxtimes H$\;
}{
    \eIf{$R$ is active movement rules}{
        $G = H \times H$\;
    }{
        $G = H \BoxProduct H$\;
    }
}
\BlankLine
\tcc{create the corresponding induced subgraph}
$I=\emptyset$\;
\ForEach{$(u,v) \in V(G)$}{
    \If{$d_H(u,v) \geq D$}{
        $I = I \cup \{(u,v)\}$\;
    }
}
$G_I = \langle I \rangle_G$\; \label{alg51-l12}
\BlankLine
\tcc{check if any component projects to $V(H)$ or to $H$}
\ForEach{component $C$ of $G_I$}{\label{alg51-l13}
    \If{$C$ projects to $V(H)$ or to $H$}{\label{alg51-l14}
        \Return{\True}
    }
}
\Return{\False}\label{alg51-l16}
\caption{existsSafeWalkAtGivenDistance($H$, $D$, $R$)\label{algCheckWalkD}}
\end{algorithm}

\begin{theorem}\label{thm:algD}
Algorithm \ref{algCheckWalkD} correctly identifies whether two players can traverse all vertices/edges of a given graph $H$ whilst maintaining a safety distance at least $D$ in polynomial time.
\end{theorem}

\begin{proof}
Depending on the movement rules $R$, we use the results from Theorems \ref{thm:s-vertexSpanOnStrongProduct}, \ref{thm:s-edgeSpanOnStrongProduct}, \ref{thm:ev-spanOnDirectProduct} and \ref{thm:s-spanOnCartesianProduct} which imply that it suffices to check only subgraphs of the corresponding products. Denote the corresponding product by $\star$. Moreover, if there exists a connected subgraph $Z \subseteq_C H \star H$ with the (edge) surjective projections such that $\varepsilon_{H}(Z)\geq D$, then for every vertex $(u,v) \in V(Z)$ it holds true that $d_H(u,v) \geq D$ (the distance condition). Lines \ref{alg51-l1}-\ref{alg51-l12} of Algorithm \ref{algCheckWalkD} compute the maximal induced subgraph $G_I$ of $H \star H$ in which for every vertex the distance condition is true. Clearly, $Z$ is a subgraph of some component of $G_I$. In lines \ref{alg51-l13}-\ref{alg51-l16} we check whether a component $C$ projects to $V(H)$ (in case players need to traverse all vertices) or to $H$, when the players must visit all edges of $H$. If such a component exists, then the whole component can be taken as the graph $Z$. If no such component exists, the players cannot maintain the desired distance at least $D$.

Denote by $n$ the number of vertices of the graph $H$, also assume the distances between any two vertices of $H$ have been previously determined. The desired product $H\star H$ and the subgraph $G_I$ can be determined in $O(n^4)$ time. Also, the surjectivity of the first (second) projection of a component $C$ to $V(H)$ can be checked by a simple loop checking whether every vertex of $H$ appears as the first (second) coordinate of some vertex of $G_I$. In the worst case scenario, all components of $G_I$ must be checked, amounting to checking at most $O(n^2)$ vertices. To check the edge surjectivity of the first (second) projection of a component $C$ to $H$, for every edge of $G_I$ we can label the edge of $H$ to which the first (second) projection maps to. If for both projections there are no unlabelled edges of $H$, then the projections are edge surjective. This can be done in $O(n^4)$ time. This concludes the proof.
\end{proof}

\begin{algorithm}[!ht]
\SetKwData{True}{true}
\SetKwData{False}{false}
\SetKw{KwDownTo}{down to}
\SetKwFunction{KwAlg}{existsSafeWalkAtGivenDistance}
\KwIn{graph $H$, movement rules $R$}
\KwOut{the maximal safety distance two players can maintain at all times while traversing all vertices/edges the graph $H$ under the movement rules $R$}
\DontPrintSemicolon
\BlankLine
\For{$i = \rad(H)$ \KwDownTo $1$}{
    \If{\KwAlg{$H, i, R$} == \True}{
        \Return $i$\;
    }
}
\Return 0\;
\caption{span($H$, $R$)\label{algSpan}}
\end{algorithm}

\begin{theorem}\label{thm:algComputeSpan}
Algorithm \ref{algSpan} returns the span of the given graph $H$ which corresponds to the given movement rules $R$ in polynomial time.
\end{theorem}

\begin{proof}
The assertion follows immediately using Lemma \ref{mLessRad} which gives $\rad(H)$ as the upper bound for any of the defined spans, and Theorem \ref{thm:algD}. At each step of the loop in Algorithm \ref{algSpan} the Algorithm \ref{algCheckWalkD} is called for the appropriate value. Since we are looking for the maximal possible safety distance, we can start from the upper bound and check down to 0, thus we can stop the first time Algorithm \ref{algCheckWalkD} returns true. Since the radius of $H$ of order $n$ is bounded from above by $n$, the for loop of the algorithm is executed at most $n$ times, using the fact that Algorithm \ref{algCheckWalkD} is polynomial, it follows that the corresponding span can also be determined in polynomial time.
\end{proof}

\section{Open problems}
In addition to some open problems stated previously, we conclude the paper with the following open problems.

In Section \ref{tri} we characterise 0-span graphs for all possible variants. Also, we give infinite families of graphs for which the chosen span is 1. The following question is a natural generalisation of these results. Note, the same can be asked for any of the defined spans.

\begin{problem}\label{p1}
Let $n$ be a positive integer. Characterise all connected graphs $H$ (or find non-trivial families of graphs $H$) with the strong vertex span $\svSpan{H}=n$.  
\end{problem}

Towards the solution of Problem \ref{p1}, for each positive integer $n$, we now define a non-empty family $\mathcal{G}_n$ of graphs such that for each $G \in \mathcal{G}_n$ the spans of $G$ are at least $n-1$ or $n$, depending on the chosen span.

\begin{definition}\label{def:familyGn}
Let $n$ be a positive integer and let $G$ be a connected graph with $\diam(G)\geq n$. We say that $G$ is $n$-friendly if for all $u,v,w \in V(G)$, \[d(u,v)=n \text{ and } w\in N(u) \Longrightarrow \text{ there is } z\in N(v) \text{ such that } d(w,z) = n. \]  We use $\mathcal{G}_n$ to denote the family of all $n$-friendly graphs.
\end{definition}

A connected graph is called \emph{even} if, for any vertex $v \in V (G)$, there exists a unique vertex $v' \in V(G)$ such that $d(v,v') = \diam(G)$. An even graph is called \emph{harmonic-even}, if $uv\in E(G)$ whenever $u'v' \in E(G)$ for all $u,v \in V(G)$, see \cite{GoVe86, KlKo09} for results on even and harmonic-even graphs.

Note:
\begin{itemize}
    \item every harmonic-even graph $G$ is $\diam(G)$-friendly,
    \item if $n\geq 3$, then every cycle $C_k$, where $k\geq 2n$, is $n$-friendly, and
    \item if $n\geq 2$, then every hypercube $Q_k$, where $k \geq n$, is $n$-friendly. For more information about hypercubes see \cite{HIK2011knjiga}.
\end{itemize}

\begin{observation}
Let $n$ be a positive integer and $G$ be an $n$-friendly graph. Then the following hold true:
\begin{enumerate}
    \item $n \leq \seSpan{G} \leq \svSpan{G} \leq \rad(G)$,
    \item $n \leq \deSpan{G} \leq \dvSpan{G} \leq \rad(G)$, and
    \item $n-1 \leq \ceSpan{G} \leq \cvSpan{G} \leq \rad(G)$.
\end{enumerate}

Moreover, if $n = \rad(G)$, then 
\begin{enumerate}
    \item $\seSpan{G} = \svSpan{G} = \rad(G)$,
    \item $\deSpan{G} = \dvSpan{G} = \rad(G)$.
\end{enumerate}

To see that the assertions are true, it is sufficient to show that $n \leq \seSpan{G}$, $n \leq \deSpan{G}$ and $n-1 \leq \ceSpan{G}$. To show $n \leq \seSpan{G}$, $n \leq \deSpan{G}$, Alice and Bob can use the following movement strategies. Let Alice and Bob start at two vertices at distance $n$. Each time Alice moves from a vertex $u$ to an adjacent vertex $w$, Bob can move from $v$ to a corresponding vertex $z$, thus maintaining the distance before the move. This is possible since $G$ is $n$-friendly. When Alice visits all edges of $G$, they swap roles. Note that this solves the problem for traditional and active movement rules.

To show that $n-1 \leq \ceSpan{G}$, let Alice and Bob start at two vertices at distance $n$. Each time Alice moves from a vertex $u$ to an adjacent vertex $w$, Bob cannot move, therefore the distance between them is at least $n-1$. After every Alice's move, if the distance to Bob is $n-1$, then Alice does not move and Bob moves from $v$ to a corresponding vertex $z$ which is again at distance $n$ from Alice. If the distance between the players if at least $n$ after Alice's move, then Bob does not move, and Alice moves. When Alice visits all edges of $G$, they swap roles. This solves the case for lazy movement rules.
\end{observation}

Since for any graph $H$ it holds true that $\seSpan{H}\leq \svSpan{H}$ (similar result applies to all span variants), it is natural to ask the following question (also for all span variants).

\begin{problem}
What is the maximum possible difference between $\svSpan{H}$ and $\seSpan{H}$?
\end{problem}

\acknowledgements
This work was supported by the Slovenian Research Agency under the grants P1-0297, J1-1693, and the research program P1-0285.

The authors are also grateful for the reviewers' remarks and suggestions for the improvement of this paper.

\nocite{*}
\bibliographystyle{abbrvnat}
\bibliography{zbirka}
\label{sec:biblio}

\end{document}